\documentclass[12pt, twoside, reqno]{amsart}
\usepackage{amsfonts}
\usepackage{amsmath}
\usepackage{ifthen}
\usepackage{amssymb}
\usepackage{color}
\usepackage{epsf,graphicx}

\numberwithin{equation}{section}

\newtheorem{theorem}{Theorem}[section]

\newtheorem{corollary}[theorem]{Corollary}
\newtheorem{definition}[theorem]{Definition}
\newtheorem{example}[theorem]{Example}

\newtheorem{proposition}[theorem]{Proposition}

\theoremstyle{definition}

\newcommand{\C}{{\mathbb C}}
\newcommand{\D}{{\mathbb D}}
\newcommand{\R}{{\mathbb R}}
\newcommand{\B}{{\mathbb B}}
\newcommand{\Hol}{{\operatorname{Hol}\,}}
\renewcommand{\Re}{{\operatorname{Re}\,}}
\renewcommand{\Im}{\,{\operatorname{Im}\,}}

\newcommand{\N}{{\mathcal N}}

\newcommand {\Id}{\mathop{\rm Id}\nolimits}

\begin{document}

\title[Nonlinear resolvents]{Nonlinear resolvents: distortion and order of starlikeness}
\date{\today}

\author[M. Elin]{Mark Elin}
\address{Braude College of Engineering, Karmiel,  Israel}
\email{mark\_elin@braude.ac.il}


\begin{abstract}
This paper presents a new approach to studying nonlinear resolvents of holomorphically accretive mappings on the open unit ball of a complex Banach space. 
We establish a distortion theorem and apply it to address problems in geometric function theory concerning the class of resolvents. 
Specifically, we prove the accretivity of resolvents and provide estimates for the squeezing ratio. 
Further, we show that nonlinear resolvents are starlike mappings of certain order and determine lower bounds for this order.
\end{abstract}

\maketitle

\section{Introduction}\label{sect-intro}

This paper is devoted to the study of properties of nonlinear resolvents, that begun in \cite{E-S-S} in the one-dimensional case and then continued in \cite{E-24a, E-J-resolv}. Like in the theory of linear operators, nonlinear resolvents play an important role in the theory of semigroups of holomorphic mappings.  

Outstanding and stimulating results in this theory were obtained in the one-dimensional case by Berkson and Porta in \cite{B-P}, in multi-dimensional settings by Abate  \cite{Ab-92}. The general situation of semigroups acting in the open unit ball of a complex Banach space was considered in the series of works by Reich and Shoikhet, starting from  \cite{R-S-96}. Over the past decades, various characterizations of semigroup generators have been found, some of them are presented in Propositions~\ref{prop-accre}, \ref{lemma1}, \ref{propo-sque} below. For more details the books \cite{B-C-DM-book, E-R-Sbook, E-S-book, R-S1, SD} and references therein.

\subsection{State of the art and main questions}\label{ssec-what}
An effective method for exploring how specific properties of the generator influence the dynamic behavior of the generated semigroup is through the use of the so-called exponential or product formula (see, for example, formula (6.183) in \cite{R-S1}). 
This formula incorporates the so-called nonlinear resolvents $G_\lambda$ of the semigroup generator $f$, defined by the formula $G_\lambda:=(\Id + \lambda f)^{-1}$, see Definition~\ref{def-range_cond} below. 
It turns out that nonlinear resolvents are biholomorphic self-mappings of the open unit ball, see~\cite{E-R-Sbook, R-S-96, R-S1}. Thus,  besides their role in dynamical systems, resolvents represent a class of mappings that is inherently intriguing from the point of view of geometric function theory.

Although this fact has been known for a long time, the study of geometric properties of nonlinear resolvents was initiated  quite recently in \cite{E-S-S}, where important geometric properties of resolvents in the open unit disk $\D\subset\C$ were established. 
Partial generalizations to multi-dimensional settings are given  in \cite{GHK2020, HK2022}. 
Obtaining full multi-dimensional analogues is a more complicated problem, requiring a new method. 

In the one-dimensional case, the above mentioned study was continued in  \cite{E-J-resolv}, where a different approach to establishing both geometric and dynamic properties was presented. Its use made it possible to sharpen previous results, as well as to consistently establish new facts, each as a consequence of the previous one. An infinite-dimensional analog of this approach was first suggested in \cite{E-24a}.

In this paper,  we develop an approach that is a combination of one used in  \cite{E-J-resolv, E-24a} with  a different idea.  As before, the object of our study is nonlinear resolvents that are  holomorphic in the open unit ball of a complex Banach space and preserve the origin (this is equivalent to $f(0)=0$, for details see Section~\ref{sect-prelim-aux}). More precisely, we consider resolvents belonging to a class of mappings (denoted $\widehat{\Hol}(\B,X)$, see Definition~\ref{def-odt}), which in the one-dimensional case coincides with the class of all functions that vanish at the origin, but in general settings is a proper subclass of all such mappings. Various features of this class have been studied by many mathematicians. 
In particular, their geometric properties are set to \cite{Bav, D-L-14, Lic-86}. The works \cite{D-L-21, D-L-22, H23, X-X-18} are devoted to estimates of some coefficient functionals for such mappings. Semigroups generated by holomorphically accretive mappings of this class  were considered, for instance, in \cite{deFab, E-S-04}. Some additional information can be found in the book \cite{G-K}.               It can be expected that the method presented below may allow to study nonlinear resolvents that are not vanishing at the origin.   This prospect will be explored elsewhere. 

To clearly present the approach, we will focus on three specific problems. We utilize two parameters with clear dynamic meaning to distinguish particular subclasses of resolvents. The union of these subclasses represents the set of all resolvents that preserve zero. For each subclass, we will address the following questions:

\begin{itemize}
  \item [$\bullet$] Establish sharp distortion theorem for this subclass.\vspace{1mm}
    \item [$\bullet$] Establish whether resolvents are holomorphically accretive. If the answer is affirmative, find squeezing ratio for semigroups generated by them.\vspace{1mm}
  \item [$\bullet$] Find  order of starlikeness of non-linear resolvents from each specific class. 
\end{itemize} 
Regarding the last question, it is noteworthy that, as proven in \cite{E-S-S}, any resolvent in the open unit disk is a starlike function of order at least $\frac12\,.$  Furthermore, when the resolvent parameter $\lambda$ exceeds a certain critical value, it has been shown in \cite{E-J-resolv} that the order of starlikeness of nonlinear resolvents can be estimated from below by a continuous function that approaches $1$ as $\lambda\to\infty$. As for the finite-dimensional case, it was shown in \cite{GHK2020} that resolvents of holomorphically accretive mappings from the class $\widehat{\Hol}(\B,\C^n)$ are starlike but order of starlikeness was not determined.  For the infinite-dimensional case see \cite{E-24a}.

\subsection{Main results and outlines}\label{ssec-main}
 In what follows we study nonlinear resolvents of holomorphically accretive mappings $f$ on the open unit ball $\B$ of a complex Banach space. Specifically, we assume that  for some $a\ge0$ we have $\displaystyle \Re \left\langle f(x),x \right\rangle \ge a\|x\|^2$  for all $x\in\B$. 

The main result of this paper is the following distortion theorem for resolvents: 

 \begin{theorem}\label{thm-reso-general-1-a}
  Let $f$ have the form $f(x)=p(x)x$,  where $p\in\Hol(\B,\C)$ with $\Re p(x) \ge a,$  $x\in\B$, for some $a\ge0$. Let $\{G_\lambda\}_{\lambda>0}$ be the resolvent family for~$f$. Denote $q=p(0)$. Then for any $\lambda>0$ we have 
   \begin{equation}\label{ineq-reso2}
   \left\|  G_\lambda(x)  \right\|  \le  \left(\frac{2}{A+\sqrt{B}}\right)^{\frac12} \|x\|,\qquad   x\in\B,
   \end{equation}
where  $A:=\left| 1 - \lambda q  \right|^2+4\lambda a+1$ and $B:=\left(|1-\lambda q|^2-1\right)^2 +8\lambda^3a|q|^2$. 
\end{theorem}

{\it Inter alia}, this theorem enables us to show that nonlinear resolvents themselves are holomorphically accretive and to find squeezing ratio of semigroups generated by them (see Theorem~\ref{thm-class}). 

The second main result establishes order of starlikeness for resolvents. To formulate it we denote $ T(r)= \frac{2 \alpha r} {(1+\beta)(1-r)^2 +\alpha(1-r^2)}$,  $\alpha>0,\ \beta\ge0$, and $ \rho^*:= \frac{\sqrt{1+\lambda\Re q}}{\sqrt{2\lambda(\Re q-a)}+\sqrt{1+\lambda\Re q}}.$

\begin{theorem}\label{thm-1dim-a}
Under the above conditions $G_\lambda$ is a starlike mapping of order $\frac12$ for any $\lambda>0$. Moreover, if for some $\lambda>0$, $\|G_\lambda(x)\| \le \rho$ on $\B$, then
  \[
  \left| \frac{1}{\|x\|^2} \left\langle \left(G_\lambda'(x) \right)^{-1} G_\lambda(x), x^* \right\rangle -1\right| \le  T(\rho) \quad\mbox {for all}\quad  x\in\B,
 \]
where $\alpha= \lambda (\Re q-a)$ and $\beta=\lambda a$.    Consequently, if $\rho\le\rho^*$, then $G_\lambda$ is starlike of order $\frac{1}{1+T(\rho)}$ and strongly starlike of order $\frac{2\arcsin T(\rho)}\pi$.
\end{theorem}
%

In the next section, we introduce required notions including holomorphically accretive mappings and nonlinear resolvents. Section~\ref{sect-onedim} is devoted to the class $\widehat{\Hol}(\B,X)$ and its subclasses. In Sections~\ref{sect-resolv-1}--\ref{sect-starlike} we formulate and prove our main results.

\medskip

\section{Preliminary and auxiliary results}\label{sect-prelim-aux}

\setcounter{equation}{0}

Let $X$ be a complex Banach  space equipped with the norm $\|\cdot\|$, $\B_r:=\left\{x\in X:\ \|x\|<r \right\}$, and $\B:= \B_1$ be the open unit ball in $X$. We denote by $X^*$  the space of all bounded linear functionals on $X$ (the dual space to $X$)  with the duality pairing $\langle x,a\rangle,\ x\in X, a\in X^*$. For each $x\in X$, the set $J(x)$, defined~by
\begin{equation}\label{Jx-set}
J(x):=\left\{ x^{\ast }\in X^{\ast }:\ \left\langle x,x^{\ast }
\right\rangle =\left\Vert x\right\Vert ^{2}= \left\Vert x^{\ast
}\right\Vert ^{2}\right\},
\end{equation}%
is non-empty according to the Hahn--Banach theorem (see, for example, \cite[Theorem~3.2]{Rudin}). It may consists of a singleton (for instance, when $X$ is a Hilbert space), or, otherwise, of infinitely many elements. Its elements $x^* \in J(x)$ are called support functionals at the point $x$. 

Let $Y$ be another complex Banach space and $D\subset X,\ D_1\subset Y$ be domains. A mapping $f:D\to D_1$ is called holomorphic if it is Fr\'echet differentiable at each point $x\in D$ (see, for example, \cite{E-R-Sbook, G-K, R-S1}). The set of all holomorphic mappings from $D$ into $D_1$ is denoted by $\Hol(D, D_1)$ and $\Hol(D)=\Hol(D,D)$. The identity mapping $\Id$ is obviously holomorphic and belongs to $\Hol(D)$. 

\vspace{2mm}

\subsection{Holomorphically accretive mappings}
The notion of numerical range for linear operators acting on a complex Banach space was introduced by Lumer \cite{Lum} and extended to non-linear holomorphic mappings by Harris \cite{Har}. For the case of a holomorphic mapping $f$ on the open unit ball $\B$ having a uniformly continuous extension to the closed ball $\overline{\B}$, the numerical range of $f$ is defined by 
\[
V(f):=\left\{\left\langle f(x),x^*\right\rangle:\ x\in\partial\B,\  x^*\in J(x)  \right\}.
\] 

The numerical range finds numerous applications in multi- and infinite-dimensional analysis, geometry of Banach spaces and other fields (see, for example, the recent books~\cite{E-R-Sbook, G-K, R-S1}). It also serves as the base to the study of so-called holomorphically accretive mappings (not nessecarily continuous on $\overline \B$). Namely,

$\diamond$ {\it for $f\in\Hol(\B,X)$ let denote following \cite{Har, H-R-S},
\begin{equation}\label{n}
n(f):=\liminf_{s\to1^-} \inf\left\{\Re w:\ w\in V(f(s\,\cdot)) \right\}.
\end{equation}
Then $f$ is said to be holomorphically accretive if $n(f)\ge0$.}  Thus, 

\begin{proposition}\label{prop-accre}
 A mapping $f\in\Hol(\B,X),\ f(0)=0,$ is holomorphically accretive if and only if $\Re \left\langle f(x),x^*\right\rangle\ge0$ for all $x\in\B$. 
\end{proposition}
Definition~\ref{def-dissip} below presents some refinements connected to this fact.

To demonstrate the importance of holomorphically accretive mappings we recall two concepts. One says that 

$\diamond$ {\it  $f\in\Hol(\B,X)$ is an infinitesimal generator if for any $x\in\B$ the Cauchy problem
\begin{equation}  \label{semig}
\left\{
\begin{array}{l}
\displaystyle
\frac{\partial u(t,x)}{\partial t} +f(u(t,x)) =0, \vspace{2mm} \\
u(0,x)=x,%
\end{array}%
\right.
\end{equation}
has the unique solution $u(t,x)\in\B$ for all $t\ge0$.} This solution $\left\{u(t,\cdot)\right\}_{t\ge0}\subset\Hol(\B)$ forms the semigroup with respect to composition operator.

Now we define non-linear resolvents, the main object of the study in this paper.

\begin{definition}\label{def-range_cond}
Let $f\in \Hol(\B,X)$. One says that $f$ satisfies the range condition on $\B$ if $\left( \Id+\lambda f\right) (\B) \supset\B$
  for each $\lambda > 0$  and the so-called resolvent equation
\begin{equation}\label{G*}
 w+ \lambda f(w)=x
\end{equation}
has a unique solution
\begin{equation} \label{resolvent1}
w=G_{\lambda }(x)\left( =\left( \Id + \lambda f\right) ^{-1} (x) \right)
\end{equation}%
holomorphic in $\B.$ If it is the case, every mapping $G_\lambda,\ \lambda>0,$ is  called the non-linear resolvent and the family $\{G_{\lambda }\}_{\lambda \geq 0}\in \Hol(\B)$ is called the resolvent family of $f$ on $\B.$
\end{definition}

The following statement combines Theorems~6.11 and~7.5 in \cite{R-S1} with another result obtained at first in \cite{R-S-96} (see also the recent books \cite{R-S1, E-R-Sbook}).

\begin{proposition}\label{lemma1}
Let $f\in \Hol(\B,X)$. Then
\begin{itemize}
  \item the mapping $f$ is holomorphically accretive if and only if it satisfies the range condition on $\B$.
  \item If, in addition, $f(0)=0$  and $f$ is bounded on each subset strictly inside $\B$, then $f$ is holomorphically accretive if and only if it is  an infinitesimal generator.
\end{itemize}
\end{proposition}

The set of all holomorphically accretive mappings having an isolated null at the origin is often denoted by $\N_0$ (see, for example,  \cite{G-K}). Hence the following clarification is natural.

\begin{definition}\label{def-dissip}
A mapping $f\in\Hol(\B,X),\ f(0)=0,$ is said to be holomorphically accretive if
\begin{equation}\label{a-condi}
  \Re\langle f(x),x^*  \rangle \ge a\|x\|^2 \quad \mbox{ for all }\quad  x\in\B \quad \mbox{ and } \quad  x^*\in J(x)
\end{equation}
 for some $a\ge0$. We denote the class of mappings that satisfy \eqref{a-condi} with a given $a\ge0$ by $\N_a$. 
 %
 \end{definition}
 Note that for any  holomorphically accretive 
 mapping~$f$, the number~$a$ can be chosen to be zero. Whence $a>0$ the mapping $f$ is named {\it strongly holomorphically accretive
 }. To explain our interest in the classes $\N_a$ with $a>0$, recall that if the semigroup generated by $f\in\N_0$ contains neither an elliptic automorphism, nor the identity mapping, then it converges to zero uniformly on any ball $\B_r,\, r<1$. It is reasonable to inquire  whether this convergence is uniform on the {\it whole ball~$\B$}. The following statement answers this question.         
 \begin{proposition}\label{propo-sque}
   Let $f$ be a holomorphically accretive mapping on $\B$ and $\left\{u(t,\cdot) \right\}_{t\ge0}$ be a semigroup generated by $f$.
   \begin{itemize}
     \item  [(i)] If $f\in\N_a,\ a>0,$ then $\left\{u(t,\cdot) \right\}_{t\ge0}$ converges to zero uniformly on $\B$ with the squeezing ratio $\kappa=a$, which means that
    \begin{equation}\label{squee}
      \left\| u(t,x) \right\| \le e^{-at} \|x\|\quad \mbox{for all }\ x\in\B.
    \end{equation}
     \item [(ii)] If $X$ is a Hilbert space and the semigroup $\left\{u(t,\cdot) \right\}_{t\ge0}$ satisfies inequality \eqref{squee} for some $a>0$, then $f\in\N_a$.
   \end{itemize}
 \end{proposition}
 Assertion (i) follows from \cite[Lemma~3.3.2]{E-R-S-04}, where we set $\alpha(s)=as$, while assertion (ii) is proved in \cite{E-R-S-02}, see also \cite{E-R-Sbook, R-S1}. 
 \vspace{2mm}

\subsection{Starlike mappings
}
Holomorphically accretive mappings play an essential role not only in dynamical systems, but also in geometric function theory (the  reader is referred to \cite{E-R-Sbook, G-K, R-S1}).  

Following \cite{G-K}, see also \cite{E-R-S-04, HKL, KL, STJ-77}, we recall several notions. A mapping $f\in\Hol(D,D_1)$ is said to be biholomorphic  if it is  invertible and $f^{-1}\in\Hol(D_1,D)$.

\begin{definition}\label{def-star}
  Let  $f\in\Hol(\B,X),\ f(0)=0,$ be a biholomorphic mapping. It is called starlike if
  \begin{equation}\label{starl}
\Re  \left\langle   \left( f'(x) \right)^{-1} f(x), x^* \right\rangle\ge0
\end{equation}
for all $x\in\B,\ x^*\in J(x)$. Moreover, the mapping $f$ is called starlike of order $\gamma\in(0,1]$ if it satisfies
\begin{equation}\label{star-order}
  \left|   \frac{\left\langle   \left(f'(x) \right)^{-1} f(x), x^* \right\rangle}{\|x\|^2} - \frac{1}{2\gamma} \right| \le \frac{1}{2\gamma},\qquad x\in\B\setminus\{0\}.
\end{equation}
It is called strongly starlike of order $\beta\in[0,1]$ if 
\begin{equation}\label{str-star-order}
  \left| \arg \left\langle  \left( f'(x) \right)^{-1} f(x), x^* \right\rangle \right| \le \frac{\pi\beta}{2},\qquad x\in\B.
\end{equation}
\end{definition}
In the one-dimensional case these definitions coincide with classical ones, see, for example, books \cite{Dur, E-S-book, G-K}, where more details can be found. Starlike mappings of order $\gamma=0$ (as well as strongly starlike of order $\beta=1$) are just starlike, while the only mappings starlike of order $\gamma=1$ (strongly starlike of order $\beta=0$) are linear ones.
Notice that the normalization $f'(0)=\Id$ is usually imposed. Since the relations~\eqref{starl}--\eqref{str-star-order} are invariant under the transformation $f\mapsto Bf$, where $B$ is an invertible bounded linear operator, we conclude that the normalization is unnecessary.

\medskip

\section{The class $\widehat{\Hol}(\B,X)$}
\label{sect-onedim}
\setcounter{equation}{0}



In this paper we focus on a class of holomorphic mappings which in the one-dimensional case coincides with the class of all holomorphic functions in the open unit disk vanishing at zero.  

\begin{definition}\label{def-odt}
  We write  $f\in \widehat{\Hol}(\B,X)$ if $f$ has the form $f(x)=p(x)x,$  where $p\in\Hol(\B,\C)$, that is,  for every point $x\in\B$, the vectors $x$ and $f(x)$ are $\C$-proportional.   
  We also denote $\widehat{\N}_a:=\N_a\cap \widehat{\Hol}(\B,X)$  and  $\widehat{\Hol}(\B):=\Hol(\B)\cap  \widehat{\Hol}(\B,X)$.
\end{definition}


Writing $f\in\widehat{\Hol}(\B,X)$ in the above form $f(x)=p(x)x$, we get $\langle f(x),x^*\rangle=  \|x\|^2 p(x)$ and $f'(0)=p(0)\Id$. 
Hence $f\in\widehat{\N}_a$, see Definition~\ref{def-dissip}, if and only if $\Re p(x)\ge a,\ x\in\B$. In the next proposition we collect some properties of such functions $p$. We provide their proof for the sake of completeness. Alternatively, they can be obtained from well-known one-dimensional results (cf. \cite{Dur, E-R-Sbook, G-K}).
 
  \begin{proposition} \label{prop-ineq}
   Let $p\in\Hol(\B,\C),$  $q=p(0)$ and $a\ge0$. The following assertions are equivalent:
    \begin{itemize}
      \item [(i)] $\Re p(x)\ge a$ for all $x\in\B$;

     \item [(ii)] (the Riesz--Herglotz formula) for every $u\in\partial\B$ there is a probability measure $\mu=\mu_u$ on the unit circle 
         such that for all $z\in\D$
         \begin{equation*}
p(zu)=\int_{|\zeta|=1} \frac{(1+z\overline{\zeta})(\Re q-a)}{1-z\overline{\zeta}}\,d\mu(\zeta)+a+i\gamma,\quad \gamma=\Im q;
\end{equation*}

      \item [(iii)] for every $x\in\B$ the value $p(x)$ lies in the closed disk centered at the point $\displaystyle c(x):= \frac{p(0) +\|x\|^2 \overline{p(0) } - 2a\|x\|^2}{1-\|x\|^2} $ and of radius $\displaystyle r(x):=\frac{2\|x\|\left( \Re p(0)  -a \right)}{1-\|x\|^2}$;

      \item [(iv)]  (Harnack's type inequality)  for every $x\in\B$ the following inequality holds
   \[
 \frac{(1-\|x\|)\Re p(0) + 2a\|x\|}{1+\|x\|}  \le \Re p(x) \le  \frac{(1+\|x\|) \Re p(0)  - 2a\|x\|  }{1-\|x\|}  .
   \]
    \end{itemize}
 \end{proposition}
 
 \begin{proof}
    Let  $\Re p(x)\ge a$ in $\B$. Fix arbitrary $u\in\partial\B$ and consider the function $\widetilde{p}\in\Hol(\D,\C)$ defined by
    \begin{equation}\label{q}
    \widetilde{p}(z):=  \frac{p(zu)-a-i\gamma }{\Re q-a}\,,\qquad \gamma=\Im q.
    \end{equation}
It has positive real part in the open unit disk $\D$ and satisfies $ \widetilde{p}(0)=1$. Therefore by the classical Riesz--Herglotz formula  there is a probability measure $\mu$ on the unit circle such that
\[
\widetilde{p}(z)= \int_{|\zeta|=1} \frac{1+z\overline{\zeta}}{1-z\overline{\zeta}}\,d\mu(\zeta).
\]
This representation is equivalent to assertion (ii).

Assume now that assertion (ii) holds. For $x\in\B$ we set $u=\frac{1}{\|x\|}x$ and define $c(x)$ and $r(x)$ as in assertion (iii). Then
\begin{eqnarray*}
  p(x) - c(x) &=& p(\|x\|u) - c(x) \\ 
                    &=& \int_{|\zeta|=1}  \frac{(1+\|x\|\overline{\zeta})(\Re q-a)}{1-\|x\|\overline{\zeta}}\, d\mu(\zeta)+a+i\gamma\\
                  &-& \frac1{1-\|x\|^2} \left( p(0) +\|x\|^2 \overline{p(0)} - 2a\|x\|^2 \right) \\
                  &=& (\Re q-a) \int_{|\zeta|=1} \left[\frac{1+\|x\|\overline{\zeta}}{1-\|x\|\overline{\zeta}} -\frac{1+\|x\|^2}{1-\|x\|^2} \right]d\mu(\zeta)\\
                  &=& (\Re q-a) \int_{|\zeta|=1} \frac{2(\|x\|\overline{\zeta}-\|x\|^2)}{(1-\|x\|\overline{\zeta})(1-\|x\|^2)} d\mu(\zeta) .
\end{eqnarray*}
Hence 
\begin{eqnarray*}
  \left|  p(x) - c(x) \right| &\le & (\Re q-a) \int_{|\zeta|=1}  \frac{2\|x\|\left|\overline{\zeta}-\|x\|\right|}{\left|1-\|x\|\overline{\zeta}\right|(1-\|x\|^2)}d\mu(\zeta) \\
   &= & (\Re q-a) \int_{|\zeta|=1}  \frac{2\|x\|}{1-\|x\|^2}d\mu(\zeta) =r(x),
\end{eqnarray*}
that is, (iii) holds.

On the other hand, if assertion (iii) holds, then
\[
\Re c(x) -r(x) \le \Re (p(x)) \le  \Re c(x) + r(x),
\]
which gives assertion (iv).

Assume now that assertion (iv) holds. The using the maximum principle leads to
\[
\Re p(x) \ge \lim_{\|x|\to1^-} \left(  \Re p(0) \frac{1-\|x\|}{1+\|x\|} +\frac{2a\|x\|}{1+\|x\|} \right)  =a,
\]
that is, assertion (i) holds. The proof is complete.
   \end{proof}
   
   Our next result is a refinement of Proposition~\ref{lemma1} for the class $\widehat{\N}_0$.

\begin{proposition}\label{lemma2a}
Let $f\in \widehat\N_0$, that is, $f(x)=p(x)x,$ and let $\{G_{\lambda }\}_{\lambda\ge0}\in \Hol(\B)$ be its resolvent family.~Then
\begin{itemize}
    \item [(a)] $\{G_{\lambda }\}_{\lambda > 0}\subset \widehat{\Hol}(\B)$, that is, for every $\lambda\ge0$ there is $g_\lambda\in\Hol(\B,\C)$ such that $G_\lambda(x)=g_\lambda(x)x$.

    \item [(b)] for every $u\in\partial\B$, the function $\widetilde{f}_u\in\Hol(\D,\C)$ defined by $\widetilde{f}_u(z):= p(zu) z$ is holomorphically accretive on the open unit disk $\D$. The resolvent family of $\widetilde{f}_u$ is $\{g_\lambda(\cdot u)\cdot\}_{\lambda > 0}$.
\end{itemize}
\end{proposition}

\begin{proof}
First we note that the discussion before Proposition~\ref{prop-ineq} provides $\Re p(x)\ge0,\ x\in\B.$ Further, by Proposition~\ref{lemma1}, $f$ satisfies the range condition. For every resolvent $G_\lambda$ we have
  \[
  G_\lambda(x) + \lambda p(G_\lambda(x))G_\lambda(x)=x
  \]
  or, equivalently, $ \left[1 +\lambda p(G_\lambda(x))\right] G_\lambda(x)=x$, so $x$ and $G_\lambda(x)$ are proportional, that is, $G_\lambda\in\widehat{\Hol}(\B)$. Thus for every fixed $\lambda>0$ there is a function $g_\lambda\in\Hol(\B,\C)$ such that $G_\lambda(x)=g_\lambda(x)x$. Substitute this relation in the resolvent equation and get:
  \[
  g_\lambda(x)x+\lambda p\left(g_\lambda(x)x\right)g_\lambda(x)x=x.
  \]
  Take $x=zu$ with $z\in\D$ and $u\in\partial\B$:
  \[
  g_\lambda(zu)z + \lambda p\left(g_\lambda(zu)zu\right) g_\lambda(zu)z=z.
  \]
  The last equality means that the function $\widetilde{f}_u: z\mapsto  p(zu) z$ satisfies the range condition on $\D$ and $g_\lambda(zu)z$ is its resolvent. The usage of Proposition~\ref{lemma1} completes the proof.
\end{proof}

At the end of this section we present a connection between starlike mappings of the class $\widehat{\Hol}(\B,X)$ and one-dimensional starlike functions (cf. \cite{Bav, D-L-14, Lic-86}). 
\begin{proposition}\label{prop-star-1d}
Let $f\in\widehat{\Hol}(\B,X)$, that is, $f(x)=p(x)x,$ where $p\in\Hol(\B,\C).$  Then $f$ is a starlike mapping of order $\gamma\in[0,1)$ (respectively, strongly starlike of order $\beta\in[0,1]$) if and only if for every $u\in\partial\B$,  the function $h\in\Hol(\D,\C)$ defined by $h(z):=p(zu)z$ is starlike of order $\gamma$ (respectively, strongly starlike of order $\beta$) in  $\D$.
\end{proposition}

\begin{proof}
  First we note that $f'(x)y=\left( p'(x)y\right)x +p(x)y$ for any ${y\in X}$.  This relation implies $\left( f'(x)\right)^{-1}f(x)=\frac{p(x)}{p'(x)x+p(x)}\,x$.
 Also, fix $u\in\partial\B$ and consider $h(z)=p(zu)z$. It satisfies $\frac{h(z)}{zh'(z)}=\frac{p(zu)}{p'(zu)zu+p(zu)}$. Hence,
 \[
\left.  \frac{\left\langle   \left( f'(x) \right)^{-1} f(x), x^* \right\rangle}{\|x\|^2} \right|_{x=zu} = \frac{h(z)}{zh'(z)}
 \]
 for any $z\in\D$, $z\neq0$. Thus, the result follows from Definition~\ref{def-star}.  
\end{proof}

\medskip

\section{Distortion and accretivity of resolvents}\label{sect-resolv-1}

\setcounter{equation}{0}

In this section we prove the distortion theorem for nonlinear resolvents.  Difference between the approach used here and another one utilized in~\cite{E-24a, E-J-resolv} is that instead of the inverse function theorem developed in \cite{E-24a, E-S-2020a}, 
our proof will rely on Proposition~\ref{prop-ineq} above. 
 
 \begin{theorem}\label{thm-reso-general-1}
  Let $f\in\widehat\N_a$ for some $a\ge0$, that is, $f(x)=p(x)x$,  where $\Re p(x) \ge a,$  $x\in\B$. Let $\{G_\lambda\}_{\lambda>0}$ be the resolvent family for~$f$. Denote $q=p(0)$. Then for any $\lambda>0$ we have 
   \begin{equation}\label{ineq-reso2}
   \left\|  G_\lambda(x)  \right\|  \le  \left(\frac{2}{A+\sqrt{B}}\right)^{\frac12} \|x\|,\qquad   x\in\B,
   \end{equation}
where  $A:=\left| 1 - \lambda q  \right|^2+4\lambda a+1$ and $B:=\left(|1-\lambda q|^2-1\right)^2 +8\lambda^3a|q|^2$. 
\end{theorem}
  
  \begin{proof}
Given $\lambda>0$, denote $P(w)=1+\lambda p(w)$. Clearly, $\Re P(w)\ge 1+\lambda a,\ w\in\D.$ Then by Proposition~\ref{prop-ineq}, the value $P(w)$ of this function lies in the closed disk centered at the point 
\[
\displaystyle  C(w):= \frac1{1-\|w\|^2} \left[ (1-\|w\|^2)(1+2\lambda a -\lambda\overline{q})+  2\lambda(\Re q- a) \right]
\]
and of radius $\displaystyle R(w):=\frac{2\|w\|\lambda(\Re q-a)}{1-\|w\|^2}$. Since for every $x\in\B$, the value $w=G_\lambda(x)$ satisfies the equation $(1+\lambda p(w))w=x$, one concludes
\[
\|w\|=\frac{\|x\|}{|1+\lambda p(w)|} = \frac{\|x\|}{|P(w)|}< \frac{1}{|C(w)|-R(w)}\,,
\]
or equivalently, $\|w\|\cdot |C(w)| <1+\|w\|R(w)$. We now square this inequality, substitute the above formulas defining $C(w)$ and $R(w)$ inside it and simplify the obtained relation. This leads us to
\begin{equation*}
    \|w\|^2 \left\{  \left| 1+2\lambda a - \lambda q  \right|^2     + \frac{4\lambda^2(\Re q-a)a}{1-\|w\|^2} \right\}  <1.
\end{equation*}
The last inequality is equivalent to
\begin{equation}\label{ineq-z}
     - \|w\|^4\cdot \left| 1+2\lambda a - \lambda q  \right|^2    + \|w\|^2\cdot \left(  \left| 1 - \lambda q  \right|^2 +4\lambda a +1 \right) <1.
 \end{equation}

If $q$ is real and greater than $2a$, then for $\lambda=\frac{1}{q-2a}$  inequality \eqref{ineq-z} implies 
\[
\|w\| <\sqrt{ \frac{1}{ 4\lambda^2a(q-a) +1 }}=\frac{q-2a}{q}  \,.
\]
Hence by the Schwarz lemma $\|w\|\le \frac{q-2a}{q}\,|z|$, which gives \eqref{ineq-reso2} in this case.

Otherwise, $t=\|w\|^2$ satisfies the quadratic inequality 
\[
t^2- t \cdot \left(1 + \frac{4\lambda^2a(\Re q-a) +1}{\left| 1+2\lambda a - \lambda q  \right|^2}\right)  >\frac{-1} {\left|1+2\lambda a - \lambda q  \right|^2}\,.
\]
 Take into account that $t\in(0,1)$ and solve this inequality. We get 
\[
t<\frac{2}{A+\sqrt{B}}
\]
with $A$ and $B$ defined above. The use of the Schwarz lemma again completes the proof.
  \end{proof}  
  
\begin{example}\label{corol-reso-partic}
    If in the hypotheses of Theorem~\ref{thm-reso-general-1}, $\lambda_0=\frac{2\Re q}{|q|^2}$, then  
   \[
   \left\|  G_{\lambda_0}(x)  \right\|   \le\frac{1}{ \sqrt{2\lambda_0 a +1+\lambda_0|q|\sqrt{2\lambda_0 a}}}\, \|x\|,\qquad   x\in\B.
   \] 
   In particular, if $q=p(0)$ is a real number, then $ \left\|  G_{2/q}(x)  \right\|   \le\sqrt{\frac{q}{4a+q}}\,\|x\|$.
\end{example}
  
 In general, it is directly verified that if $a>0$, then $\frac{2}{A+\sqrt{B}}<1$. If $a=0$, Theorem~\ref{thm-reso-general-1}  implies the following consequence.
\begin{corollary}\label{cor-est1}
  If in hypotheses of Theorem~\ref{thm-reso-general-1},  $f\in\widehat\N_0$, then  
  \[
  \left\|G_\lambda(x)\right\| < \left\{ \begin{array}{cc}
                            \|x\| ,  & \mbox{as } \ \lambda\le \frac{2\Re q}{|q|^2} \vspace{2mm} \\
                           \frac{\|x\|}{|1-\lambda q|}\,,   &  \mbox{as }\ \lambda>\frac{2\Re q}{|q|^2} 
                           \end{array}       \right.,\quad z\in\D.
  \] 
In particular, the mapping $G_\lambda,$  $\lambda>\frac{2\Re q}{|q|^2}$, has no boundary fixed~points.
\end{corollary}

\begin{figure}[t]
  \includegraphics[width=7cm, height=5cm]{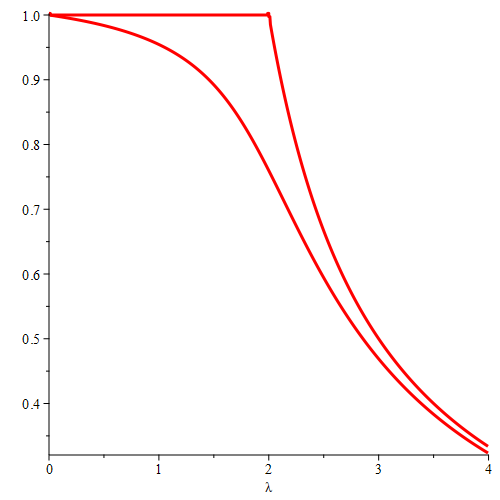}
  \caption{Function $\lambda\mapsto\sqrt{\frac{2}{A+\sqrt{B}}}$ }
  \label{fig:depend-a}
\end{figure} 

In Fig.~\ref{fig:depend-a}, the graph of the function $\lambda\mapsto\sqrt{\frac{2}{A+\sqrt{B}}}$ is presented for the values of the parameters 
$q=1$,  $a=0$ and $q=1$, $a=\frac{1}{40}$, respectively.

\vspace{3mm}

To proceed, recall that by Definition~\ref{def-range_cond},
\[
 G_\lambda(x) + \lambda f(G_\lambda(x)) =x,
\]
hence
\[
 \lambda \left\langle f(G_\lambda(x)), x^*\right\rangle  = \|x\|^2 - \left\langle G_\lambda(x), x^*\right\rangle .
\]
Therefore, estimate~\eqref{ineq-reso2} implies that
\[
 \lambda \Re \left\langle f(G_\lambda(x)), x^*\right\rangle \ge\|x\|^2 \left( 1 -  \left(\frac{2}{A+\sqrt{B}}\right)^{\frac12} \right) .
\]
So, we conclude the following:
\begin{corollary}\label{cor-generator}
 Let $f\in\widehat{\N}_a$ for some $a\ge0$ and $\{G_\lambda\}_{\lambda>0}$ be the resolvent family for~$f$.  Then
 \[
 f\circ G_\lambda \in\N_{a_\lambda},\quad \mbox{ where  }\quad  a_\lambda=\frac1\lambda\left(1-  \left(\frac{2}{A+\sqrt{B}}\right)^{\frac12} \right).
 \]
 \end{corollary}
 
 Thus $ f\circ G_\lambda $ is a holomorphically accretive mapping that generate the semigroup with squeezing ratio $a_\lambda$. 
 It turns out that resolvents themselves are holomorphically accretive. In the next theorem we prove this fact along estimation of squeezing ratio of generated semigroups. 

\begin{theorem}\label{thm-class}
  Let  $f\in\widehat{\N}_a,\ f(x)=p(x)x,$ $q=p(0)$ and $G=G_\lambda$ is the resolvent of $f$ corresponding to the parameter $\lambda>0$. Denote
 \begin{equation}\label{phi}
   \phi(t):=\frac{1+\lambda\Re q-t\left(1-\lambda(\Re q-2a)\right)} {|1+\lambda q|^2-t\left| 1-\lambda (q-2a)\right|^2}\,.
 \end{equation}
 Then $G\in\widehat{\N}_{d_\lambda}$,  $d_\lambda=\min\left\{\frac{1+\lambda\Re q} {|1+\lambda q|^2}, \phi\left(\frac2{A+\sqrt{B}} \right) \right\}$, where   $A=\left| 1 - \lambda q  \right|^2+4\lambda a+1$ and $B=\left(|1-\lambda q|^2-1\right)^2 +8\lambda^3a|q|^2$  as above. 
\end{theorem}

\begin{proof} 
It follows from Proposition~\ref{prop-ineq} that for every $w\in\B$ the value $1+\lambda p(w)$  lies in the closed disk centered at $1+\lambda c(w)$ and of radius $\lambda r(w)$ with $c(w)= \frac{ q  + \overline{q}\|w\|^2 - 2a\|w\|^2}{1-\|w\|^2}$ and $ r(w)=\frac{2\|w\|\left( \Re q  -a \right)}{1-\|w\|^2}$.
 
 Proposition~\ref{lemma2a} states that for each $\lambda>0$ the mapping $G$ is o.d.t., so that $G(x)=g(x)x$. Moreover, $g(x)=\frac1{1 +\lambda p(w)}$, where $w=G(x)$. Thus, the value $g(x)$ lies in the disk centered at the point $
 c_1(w)=\frac{1+\lambda\overline{c(w)}}{|1+\lambda c(w)|^2-\lambda^2r(w)^2}$ of radius $r_1(w)= 
 \frac{\lambda r(w)}{|1+\lambda c(w)|^2-\lambda^2r(w)^2}$.  Summarizing these facts, one concludes that
 \[
 \Re\frac{\left\langle G(x),x^*\right\rangle} {\|x\|^2} \ge \Re c_1(w)-r_1(w)= \frac{1+\lambda(\Re c(w)-r(w))}{|1+\lambda c(w)|^2-\lambda^2r(w)^2}\,, 
 \]
 where $ w=G(x).$
 
Substituting in the last expression $c_1(w)$ and $r_1(w)$, we see that
\[
\Re\frac{\left\langle G(x),x^*\right\rangle} {\|x\|^2} \ge \frac{1+\lambda\Re q-t\left(1-\lambda(\Re q-2a)\right)} {|1+\lambda q|^2-t\left| 1-\lambda (q-2a)\right|^2}\,, \ t=\|w\|^2.
\]
It follows from Theorem~\ref{thm-reso-general-1} that $t<\frac{2}{A+\sqrt{B}}$.
Thus,  $G\in\widehat{\N}_{d_\lambda}$ with $d_\lambda= \min_{t\in \left[0, \frac{2}{A+\sqrt{B}}\right] } \phi(t),$ where function $\phi$ is defined by \eqref{phi}. Since this function is monotone, its minimum is attained at the ends of the segment. So, the statement follows. 
\end{proof}

\medskip

\section{Order of starlikeness of resolvents}\label{sect-starlike}

\setcounter{equation}{0}

We  already mentioned that in the one-dimensional case nonlinear resolvents are starlike functions of order at least $\frac12$, see Subsection~\ref{ssec-what} and Definition~\ref{def-star}.    

Another method to estimate order of starlikeness of nonlinear resolvents as a function of the resolvent parameter was used in \cite{E-J-resolv} for the one-dimensional case and in \cite{E-24a} for the infinite-dimensional case. The purpose of this section is to improve and modify the method used earlier in order to find sharper order of starlikeness for nonlinear resolvents of an arbitrary mapping $f\in\widehat{\N}_a$. 

Recall that according to Definition~\ref{def-dissip}, $f(x)=p(x)x$, where $p\in\Hol(\B,\C)$ satisfies $\Re p(x)\ge a$ in $\B$. As before, let $q=p(0)$.

We need the following notation. Let $\alpha>0,\ \beta\ge0$. Denote 
  \begin{equation}\label{T}
    T(r)= \frac{2 \alpha r} {(1+\beta)(1-r)^2 +\alpha(1-r^2)}.
  \end{equation}
Obviously, $T\left(= T_{\alpha,\beta}\right)$ is an increasing function in $r\in[0,1)$. 

\begin{theorem}\label{thm-1dim-a}
  Let $f\in\widehat{\N}_a,\ a\ge0,$  and $G(=G_\lambda)$ be the resolvent of $f$ corresponding to the parameter $\lambda>0$. If $ \|G(x)\| \le \rho$ on $\B$, then 
  \[
  \left| \frac{1}{\|x\|^2} \left\langle \left(G'(x) \right)^{-1} G(x), x^* \right\rangle -1\right| \le  T(\rho) \quad\mbox {for all}\quad  x\in\B,
 \]
where function $T$ is be defined by \eqref{T} with $\alpha= \lambda (\Re q-a)$ and $\beta=\lambda a$.     
\end{theorem}

\begin{proof}
By Proposition~\ref{lemma2a}, $G\in\widehat{\Hol}(\B)$, so that $G(x)=g(x)x$ with $g\in\Hol(\B,\D)$. Therefore 
$G'(x)x=\left(g'(x)x+g(x)\right)x$, so that $$\left(G'(x)\right)^{-1}G(x)=\frac{g(x)}{g'(x)x+g(x)}\,x.$$

Let $f(x)= p(x)x$. A straightforward differentiation of the resolvent equation $G(x)+\lambda p(G(x))G(x)=x $ yields
\[
\frac{1}{g'(x)x+g(x)} = 1+\lambda p(G(x)) + \lambda\left( p'(G(x))x \right) g(x).
\]
Therefore,
\begin{eqnarray*}
 && \frac{1}{\|x\|^2} \left\langle \left(G'(x) \right)^{-1} G(x), x^* \right\rangle \\ 
 &=& \frac{1}{\|x\|^2} \left\langle \bigl(1 +\lambda p(G(x)) +\lambda p'(G(x))G(x) \bigr) G(x) , x^* \right\rangle \\
   &=& \left( 1 +\lambda p'(w)w   +p(w) \right) g(x),
\end{eqnarray*}
where we denoted $w=G(x)$. 
In addition, the resolvent equation implies $g(x)=\frac{1}{1+\lambda p(w)}$. Combining these relations we conclude that 
\begin{equation}\label{starGr}
 \frac{1}{\|x\|^2} \left\langle \left(G'(x) \right)^{-1} G(x), x^* \right\rangle = 1+ \frac{\lambda p' (w)w}{1+\lambda p(w)}\,,\quad w=G(x).
\end{equation}
Therefore we need to estimate the range of  $\displaystyle \frac{\lambda w p' (w)}{1+\lambda p(w)}$ whenever ${w\in G(\B)}$. To do this, we fix an arbitrary $u\in\partial\B$ and write the Riesz--Herglotz formula for the function $p$ in the form given in assertion (ii) of Proposition~\ref{prop-ineq}:   
       \begin{equation*}
p(zu)=\int_{|\zeta|=1} \frac{(1+z\overline{\zeta})(\Re q-a)}{1-z\overline{\zeta}}\,d\mu(\zeta)+a+i\Im q. 
      \end{equation*}

Hence the function $B_\lambda$ defined by $B_\lambda(z):=\Re (1+\lambda p(zu))$ can be represented by the integral
\begin{equation*}
B_\lambda(z)=\int_{|\zeta|=1} \frac{(1+\lambda a)|1-z\overline{\zeta}|^2 + \lambda (\Re q-a)(1-|z|^2)} {|1- z\overline{\zeta}|^2}d\mu(\zeta).
\end{equation*}

In addition, differentiating $p(zu)$ with respect to $z$ we get that the function $C_\lambda(z):=\lambda z p'(zu)u $ satisfies  
\begin{eqnarray*}
\left|C_\lambda(z) \right| &\leq & \int_{|\zeta|=1} \frac{2 \lambda (\Re q-a) |z\overline{\zeta}|}{|1-z\overline{\zeta}|^2}d\mu(\zeta) \\
  &=& \int_{|\zeta|=1}  \frac{(1+\lambda a)|1-z\overline{\zeta}|^2 + \lambda (\Re q-a)(1-|z|^2)} {|1- z\overline{\zeta}|^2} \cdot   A_\lambda(z,\zeta) d\mu(\zeta),
\end{eqnarray*}
where
\begin{eqnarray*} 
  A_\lambda(z,\zeta) &:=&  \frac{2 \lambda (\Re q-a) |z\overline{\zeta}|} {(1+\lambda a)|1-z\overline{\zeta}|^2 + \lambda (\Re q-a)(1-|z|^2)} \\
   &\le& \frac{2 \lambda  (\Re q-a) |z|} {(1+\lambda a)(1-|z| )^2 + \lambda (\Re q- a)(1-|z|^2)} =T_{\alpha,\beta}(|z|),
\end{eqnarray*}
where function $T$ is defined by \eqref{T} with $\alpha= \lambda (\Re q-a)$ and $\beta=\lambda a$.

Therefore,
\begin{equation*}
 \left|  \frac{\lambda zp'(zu)u} {1+ \lambda p(zu)}\right| \leq \frac{\left|C_\lambda (z)\right|} {B_\lambda(z)}\leq T_{\alpha,\beta}(|z|).
\end{equation*}
As we noticed, $T\left(= T_{\alpha,\beta}\right)$ is an increasing function, hence  $\left|  \frac{\lambda zp'(zu)u} {1+ \lambda p(zu)}\right|\le T(\rho)$ as $|z|\le\rho$. Since $u\in\partial\B$ is arbitrary, we conclude by our assumptions that $ \left|  \frac{\lambda p'(w)w} {1+\lambda p(w)}\right| \leq T(\rho)$ for all $w\in G(\B)$. 

Comparing this result with equation \eqref{starGr}, we see that all values of $\frac{1}{\|x\|^2} \left\langle \left(G'(x) \right)^{-1} G(x), x^* \right\rangle$ lies in the disk of radius $T(\rho)$ and centered at~$1$. This completes the proof. 
\end{proof}
 
As a consequence of Theorem~\ref{thm-1dim-a}, one can get order of starlikeness and of strong starlikeness for resolvents, see Definition~\ref{def-star}. To this end denote 
  \[
       \rho^*:= \frac{\sqrt{1+\lambda\Re q}}{\sqrt{2\lambda(\Re q-a)}+\sqrt{1+\lambda\Re q}}\,,
  \]  
which is the smallest positive root of the equation $T(r)=1$. 

\begin{corollary}\label{cor-star}
Let  $G$ be the resolvent of $f\in\widehat{\N}_a,\ a\ge0,$ corresponding to the parameter $\lambda>0$. Then $G$ is a starlike mapping of order $\frac12.$ Moreover, if $ \|G(x)\| \le \rho\le\rho^*$ on $\B$,   then mapping $G$ is starlike of order $\frac{1}{1+T(\rho)}$ and strongly starlike of order $\frac{2\arcsin T(\rho)}\pi$.
\end{corollary}
\begin{proof}
 Applying Proposition~\ref{lemma2a} (i), one concludes that $G\in\widehat{\Hol}(\B)$, that is, $G(x)=g(x)x,\ g\in\Hol(\B,\D)$. In order to prove that $G$ is starlike of some order, we have to show according to Proposition~\ref{prop-star-1d} that for every $u\in\partial\B$ the function $z\mapsto g(zu)z$ is starlike of the same order. It follows from Proposition~\ref{lemma2a} (ii) that the last function is a non-linear resolvent of the function $p(zu) z$, which is holomorphically accretive in the open unit disk $\D$. Therefore the first assertion of the theorem follows from \cite[Corollary 2.6]{E-S-S}

As for the second assertion, we note that $T(\rho)\le1$, and the disk $|\zeta-1|\le T(\rho)$ is a subset of the disk $\left|\zeta-\frac1{2\gamma}\right|\le \frac1{2\gamma}$ with $\gamma=\frac{1}{1+T(\rho)}$ and of the wedge $\left| \arg\zeta\right|\le\arcsin T(\rho)$. Thus the conclusion follows from Theorem~\ref{thm-1dim-a} and Definition~\ref{def-star}.
\end{proof}

It is natural to apply Theorem~\ref{thm-1dim-a} and Corollary~\ref{cor-star} together with finding concrete $\rho\le \rho^*$ for which $ \|G(x)\| \le \rho$, $x\in\B$. To do this, we will use the results of the previous section. 

By Theorem~\ref{thm-reso-general-1}, all values of the function $G_\lambda$ lies~in the disk of 
radius $\sqrt{\frac{2}{A+\sqrt{B}}}$, where
 \begin{equation}\label{AB}
\left\{ \begin{array}{l}
   A=\left| 1 - \lambda q  \right|^2+4\lambda a+1, \\
   B=\left(|1-\lambda q|^2-1\right)^2 +8\lambda^3a|q|^2. 
 \end{array}\right.
  \end{equation}

Thus, we need conditions that ensure the inequality ${\sqrt{\frac{2}{A+\sqrt{B}}}\le \rho^*}$, where $\rho^*$ is defined before Corollary~\ref{cor-star}. This relation is equivalent to 
\begin{equation}\label{ineq-main}
2\left(\sqrt{\frac{2\lambda(\Re q-a)}{1+\lambda\Re q}} +1 \right)^2 \le A+\sqrt{B}.
\end{equation}
(So, in this case $G$ is starlike of order $\frac{1}{1+T(\rho)}$, $\rho=\sqrt{\frac{2}{A+\sqrt{B}}}$ by  Corollary~\ref{cor-star}.)

Not knowing how to handle this inequality, we will look at stronger ones. Specifically, we analyze the two cases. 

{\bf In the first one,} $|1-\lambda q|^2\ge1$,  or equivalently, 
$\lambda|q|^2\ge2\Re q$.  
In this case we decrease $A$ and $B$ replacing $|q|$ in \eqref{AB} by $\Re q$. In particular, 
$B\ge \left((1-\lambda\Re q)^2-1\right)^2 +8\lambda^3a \Re^2q.$  Introduce the parameters $x=2\lambda(\Re q-a)$
and $s=\lambda\Re q$ and consider the inequality 
\[
2\left(\sqrt{\frac{x}{1+s}} +1 \right)^2 <s^2+2s+2-2x+s\sqrt{(s+2)^2 - 4x},
\]
which is stronger than \eqref{ineq-main}. Solving this inequality we get $x<\frac{4s^2(1+s)}{(2+s)^2}. $ This means that 
\[
\Re q -a<\frac{2\lambda\Re^2 q(1+\lambda\Re q)}{(2+\lambda\Re q)^2}\,.
\]
The solution to the last inequality with respect to $\lambda$ is $\lambda>M_1(q,a),$ where
\begin{equation}\label{M1}
M_1(q,a):=\frac{\sqrt{5\Re^2 q-4a\Re q} +\Re q-2a} {(\Re q+a)\Re q}\,.
\end{equation}

\vspace{2mm}
 
 {\bf In the opposite case} $\lambda |q|^2<2\Re q$. First, assume that $q=p(0)$ is real. Then  $B=\lambda^2 q^2 \left(2- \lambda q\right)^2 +8\lambda^3a q^2$. In this situation  inequality~\eqref{ineq-main} can be solved relative to $a$:
 \[
 a>\frac{4+2\lambda q -(\lambda q)^2}{(2+\lambda q)^2}.
 \] 

In the general situation (that is, when $p(0)$ is not necessarily real), we cannot omit $\Im q$. Therefore let estimate $B$ differently: 
$\sqrt{B}\ge 1-|1-\lambda q|^2$. Then $A+\sqrt{B}\ge 4\lambda a+2$. Thus, we replace inequality~\eqref{ineq-main} with more stringent:
\[
\left(\sqrt{\frac{2\lambda(\Re q-a)}{1+\lambda\Re q}} +1 \right)^2 < 2\lambda a+1.
\]
This inequality holds if and only if $a>M_2(q,\lambda)$, where
\begin{equation}\label{M3}
M_2(q,\lambda):=\frac{(s+1)\sqrt{2s^2 + 4s +1}+s^2 +s-1} {\lambda(2+s)^2}\,,\quad s=\lambda\Re q.
\end{equation}

We summarize these facts in the following theorem.

\begin{theorem}\label{thm-calc}
 Let $f\in\widehat{\N}_a,\ a\ge0,\ f(x)=p(x)x,\ q=p(0),$  and $G$ be the resolvent of $f$ corresponding to the parameter 
 $\lambda>0$. Let $A,\ B,\ M_1(q,a)$ and $M_2(q,\lambda)$ be defined by \eqref{AB}, \eqref{M1} and \eqref{M3}. 
 Assume that one of the following conditions holds:
 \begin{itemize}
   \item [(i)]  $\lambda|q|^2\ge2\Re q$ and $\lambda>M_1(q,a);$
   \item [(ii)]  $\lambda|q|^2<2\Re q$ and $a>M_2(q,\lambda)$.
 \end{itemize}
 Then $G$ is a starlike mapping of order $\frac{1}{1+T\left(\sqrt{\frac{2}{A+\sqrt{B}}}\right)}.$
\end{theorem}

\begin{figure}[t]
  \includegraphics[width=60mm, height=4.8cm]{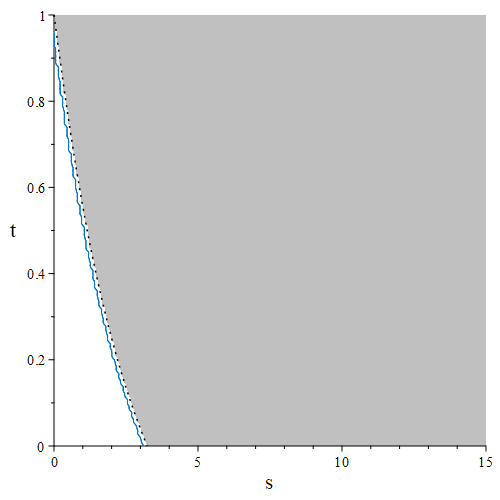}
  \caption{Range of the parameters $s$ and $t$}
  \label{fig:orders2}
  \end{figure}
  
\begin{corollary}
  If $q$ is real, $a\in\left[0,q\right)$, then $G_\lambda$ is starlike of order $\frac{1}{1+T\left(\sqrt{\frac{2}{A+\sqrt{B}}}\right)}$ 
  whenever $\lambda>\frac{\sqrt{5 p^2(0)-4a q} + q-2a}{( q+a) q}$ (cf. \eqref{M1}).
\end{corollary}
It is interesting to determine a set of parameters $\left(q,a,\lambda\right)$ for which we found the order of starlikeness.  
For simplicity, let us take the case $q\in\R$ and  depict this set by passing to the parameters $s=\lambda q>0$ 
and $t=\frac{a}{q}\in[0,1)$. In these new parameters, the desired  set can be described as $\left\{(s,t):\ s>0,\ \frac{4+2s-s^2}{(2+s)^2}<t<1
\right\}$, see  Fig.~\ref{fig:orders2}. In the one-dimensional case,  the order of starlikeness was earlier found for $t=0$ and $s\ge r_0\approx5.9$ in \cite{E-J-resolv}. 


\medskip

\begin{thebibliography}{99}

\bibitem{Ab-92} M. Abate,
The infinitesimal generators of semigroups of holomorphic maps, {\it Ann. Mat. Pura Appl.} {\bf161} (1992), 167--180.





\bibitem{Bav} I. I. Bavrin, On a class of univalent mappings in Banach spaces, \textit{Acad. Sci. Georgian SSR} {\bf 97} (1980), 285--288 (in Russian).



\bibitem{B-P} E. Berkson and H. Porta,
Semigroups of analytic functions and composition operators, {\it Michigan Math. J.} \textbf{25} (1978), 101--115.


\bibitem{B-C-DM-book} F. Bracci, M.~D. Contreras and S. D\'{\i}az-Madrigal,
{\sl Continuous Semigroups of holomorphic self-maps of the unit disc}, Springer Monographs in Mathematics, Springer, 2020.




\bibitem{deFab} C. de Fabritiis, On the linearization of a class of semigroups on the unit ball of $\C^n$, {\it Annali di Mat. Pura ed Appl.} {\bf 166} (1984), 363--379.


\bibitem{D-L-14} R. D{\l}ugosz and P. Liczberski, New properties of Bavrin's holomorphic functions on Banach spaces, {\it Bull. Soc. Sci. Lett. {\L}ódź S\'{e}r. Rech. D\'{e}form.} {\bf 64} (2014), 57--70.


\bibitem{D-L-21} R. D{\l}ugosz and P. Liczberski, 
Some results of Fekete--Szeg\"o type for Bavrin's families of holomorphic functions in $\C^n$, {\it Ann. Mat. Pura Appl. } {\bf 4} 200 (2021), 1841--1857,  https://doi.org/10.1007/s10231-021-01094-6.


\bibitem{D-L-22} R. D{\l}ugosz and P. Liczberski, 
Fekete–Szegö problem for Bavrin’s functions and close-to-starlike mappings in $\C^n$, {\it Anal. Math. Phys.}, {\bf 12} (2022), https://doi.org/10.1007/s13324-022-00714-5.



\bibitem{Dur} P. L. Duren, Univalent Functions, Springer-Verlag, New York, Berlin, Heidelberg, Tokyo, 1983.








\bibitem{E-24a} M. Elin, Non-linear resolvents of holomorphically accretive mappings, available in: /https://arxiv.org/pdf/2406.02758.
 


\bibitem{E-J-resolv}	M. Elin and F. Jacobzon, Nonlinear resolvents in the unit disk: geometry and dynamics, {\it Journ. Math. Sci.} {\bf280} (2024), 933--946. 


\bibitem{E-R-S-02} M. Elin, S. Reich and D. Shoikhet,
Asymptotic behavior of semigroups of $\rho$-non-expansive and holomorphic mappings on the Hilbert ball, {\it Ann. Mat. Pura Appl.} {\bf 181} (2002), 501--526.

\bibitem{E-R-S-04} M. Elin, S. Reich and D. Shoikhet,
Complex Dynamical Systems and the Geometry of Domains in Banach Spaces, {\it Dissertationes Math. (Rozprawy Mat.)} {\bf 427} (2004), 62 pp.


\bibitem{E-R-Sbook} M. Elin, S. Reich and D. Shoikhet,
\emph{Numerical Range of Holomorpic Mappings and Applications}, Birkh\"{a}user, Cham, 2019.


\bibitem{E-S-04} M. Elin and D. Shoikhet,
Semigroups of holomorphic mappings with boundary fixed points and spirallike mappings, in: {\sl Geom Funct. Theory Several Complex Var.},  World Sci. Publishing, River Edge, NJ.,  82--117, 2004.  
 

\bibitem{E-S-book} M. Elin and D. Shoikhet,
{\sl Linearization Models for Complex Dynamical Systems. Topics in univalent functions, functions equations and semigroup theory},
Birkh{\"a}user Basel, 2010.


\bibitem{E-S-2020a} M. Elin and D. Shoikhet,
A sharpened form of the inverse function theorem, {\it Annales UMCS}, {\bf LXXIV} (2020), 59--67.


\bibitem{E-S-S} M. Elin, D. Shoikhet and T. Sugawa,
Geometric properties of the nonlinear resolvent of holomorphic generators, {\it J. Math. Anal. Appl.} {\bf 483} (2020), https://doi.org/10.1016/j.jmaa.2019.123614.






\bibitem{G-K} I. Graham and G. Kohr,
{\sl Geometric Function Theory in One and Higher Dimensions}, Marcel Dekker, Inc, NY-Basel, 2003.


\bibitem{GHK2020} I. Graham, H. Hamada and G. Kohr, Loewner chains and nonlinear resolvents of the Carath\'{e}odory family on the unit ball in $\C^n$, {\it J. Math. Anal. Appl.} {\bf 491}  (2020), https://doi.org/10.1016/j.jmaa.2020.124289.


\bibitem{H23} H. Hamada, 
Fekete--Szeg\"o\  problems for spirallike mappings and close-to-quasiconvex mappings on the unit ball of a complex Banach space, {\it Results Math.} {\bf78} (2023). https://doi.org/10.1007/s00025-023-01895-6.


\bibitem{HK2022} H. Hamada and G. Kohr, 
Loewner PDE, inverse Loewner chains and nonlinear resolvents of the Carathéodory family in infinite dimensions,  {\it Ann. Sc. Norm. Super. Pisa Cl. Sci. (5)}, {\bf XXIV} (2023), 2431--2475.          https://doi.org/10.2422/2036-2145.202107-004.



\bibitem{HKL} H. Hamada, G. Kohr, and P. Liczberski,
Starlike mappings of order $\alpha$ on the unit ball in complex Banach spaces,
{\it Glassnik Mat.}  {\bf 36(56)} (2001), 39--48.


\bibitem{Har} L. A. Harris, The numerical range of holomorphic functions in Banach spaces, {\it Amer. J. Math.} {\bf 93}  (1971), 1005--1019.


\bibitem{H-R-S} L. A. Harris, S. Reich and D. Shoikhet,
Dissipative holomorphic functions, Bloch radii, and the Schwarz Lemma, {\it J. Analyse Math.} {\bf 82} (2000), 221--232.


\bibitem{KL} G. Kohr and P. Liczberski, On strongly starlikeness of order alpha in several complex variables, {\it  Glassnik Mat.},  {\bf 33} (1998),  185--198.


\bibitem{Lic-86} P. Liczberski,
On the subordination of holomorphic mappings in $\C^n$, {\it Demonstratio Math.} {\bf 19} (1986), 293--301.



\bibitem{Lum} G. Lumer,
Semi-inner-product spaces, {\it Trans. Amer. Math. Soc.} {\bf 100} (1961), 29--43.









\bibitem{R-S-96}  S. Reich and D. Shoikhet,
Generation theory for semigroups of holomorphic mappings in Banach spaces, {\it Abstr. Appl. Anal.} {\bf 1} (1996), 1--44.



\bibitem{R-S1} S. Reich and D. Shoikhet,
{\sl Nonlinear Semigroups, Fixed Points, and the Geometry of  Domains in Banach Spaces}, World Scientific Publisher, Imperial  College Press, London, 2005.


\bibitem{Rudin} W. Rudin, {\sl Functional Analysis}, Int. Ser. in Pure and Appl. Math. {\bf 8} NY: McGraw-Hill, 1991.


\bibitem{SD} D. Shoikhet, {\sl Semigroups in Geometrical Function Theory}, Kluwer,
Dordrecht, 2001.





\bibitem{STJ-77} T. J. Suffridge,
Starlikeness, convexity and other geometric properties of holomorphic maps in higher dimensions, {\it Complex Analysis (Proc. Conf. Univ. Kentucky, Lexington, KY, 1976), Lecture Notes in Math.} {\bf 599} (1977), 146--159.


\bibitem{X-X-18} Q. H. Xu and X. Xu, 
On the coefficient inequality for a subclass of strongly starlike mappings of order $\alpha$ in several complex variables, {\it Results Math.} {\bf 73} (2018), 1--16, https://doi.org/10.1007/s00025-018-0837-2.


\end{thebibliography}
\end{document}